\documentclass[11pt]{amsart}
\usepackage[margin=1.2in]{geometry}
\usepackage{amsmath,amssymb,amsthm}
\usepackage{booktabs}
\usepackage[colorlinks=true,linkcolor=blue,citecolor=blue,urlcolor=blue]{hyperref}

\newtheorem{theorem}{Theorem}[section]
\newtheorem{proposition}[theorem]{Proposition}
\newtheorem{lemma}[theorem]{Lemma}

\theoremstyle{definition}
\newtheorem{question}[theorem]{Question}
\newtheorem{remark}[theorem]{Remark}

\newcommand{\F}{\mathbb{F}}
\newcommand{\Z}{\mathbb{Z}}
\newcommand{\C}{\mathbb{C}}
\renewcommand{\P}{\mathsf{P}}
\newcommand{\Ball}[1]{B(#1)}
\newcommand{\supp}{\operatorname{supp}}
\newcommand{\ru}{r_{\mathrm{u}}}
\newcommand{\msupp}{m_{\mathrm{u}}}

\title[Unit-support geometry of non-unique-product groups]{Localizing the Gardam
unit: the support geometry of units in $\F_2[\P]$ and its
non-unique-product relatives}
\author{Moe Tabei}
\address{Independent researcher, Japan}
\email{tabei@ryun.jp}

\begin{document}

\begin{abstract}
Gardam's counterexample to the Kaplansky unit conjecture is a unit of
$\F_2[\P]$, $\P$ the Promislow group, with support of size $21$; Gardam asked
whether $21$ is least possible. Over $\F_2$ the unit equation is a parity
condition on a pair of supports, searchable over word balls. We upgrade two
statements to machine-checked certificates: no support
pair of sizes $\ge2$ lies in the radius-$3$ ball (a DRAT proof of a ball form of the Craven--Pappas theorem), and every
nontrivial unit with supports in the radius-$4$ ball has total support at
least $42$, confirming Gardam's expectation in ball-limited form. The
relatives $H_4=F(3,4)$ and the Nielsen--Soelberg groups are swept with
certificates; $G_3$ is Gardam's amalgam $S$, whose unit we localize into its
radius-$4$ ball. $H_4$ resists the twisted-unitary ansatz through
radius $6$ for every non-identity dihedral twist. An effective localization
principle makes the least support of a nontrivial unit of $\F_2[\P]$
computable in principle.
\end{abstract}

\maketitle

\section{Introduction}\label{sec:intro}

Three conjectures on group rings of torsion-free groups are attributed to
Kaplansky: the unit conjecture, the zero divisor conjecture and the idempotent
conjecture. The unit conjecture is now false in every characteristic: Gardam
\cite{Gardam} produced a nontrivial unit of $\F_2[\P]$, where $\P$ is
Promislow's group (the Hantzsche--Wendt group, the unique torsion-free
$3$-dimensional crystallographic group with finite abelianization; for its
continuing role in the non-unique-product landscape see also the recent
\cite{BengiWise}); Murray
\cite{Murray} extended the construction to $\F_p$ for every prime $p$, and
Gardam has since disproved the characteristic-zero case as well
\cite{GardamC}. That same paper, although its main theorem is over $\C$,
also exhibits a \emph{second} counterexample group already over $\F_2$: an
amalgam $S$ of two Klein bottle groups over $\Z^2$, whose group ring
$\F_2[S]$ carries a $29$-term twisted-unitary unit with symmetry group
$\Z/4$ \cite[Thm.~B]{GardamC}. The zero divisor and idempotent conjectures remain open in
general, but are theorems for the groups in this paper, all of which are
elementary amenable \cite{KLM}.

What remains on the unit side is quantitative. Gardam's unit has support of
cardinality $21$, and he writes that it ``appears plausible that non-trivial
units in $\F_2[\P]$ always have support of cardinality at least 21''
\cite[\S2]{Gardam}. The question of \emph{where} --- in the geometry of the
group --- the failure of the unit conjecture first appears was opened
experimentally by Dietrich--Lee--Nies--Vinyals \cite{DLNV}, who found that
the radius-$3$ ball of $\P$ carries no unit while the radius-$4$ ball
carries exactly $36$, all of which they enumerated. What has been missing
is certification, minimality thresholds, theory, and the view beyond $\P$;
those are the contributions of this paper.

This paper answers both questions in ball-limited form, with the exact-model
machinery built in \cite{companion,sequel} for the parallel study of minimal
non-unique-product (non-UP) sets. Our starting point is elementary but
consequential: over $\F_2$ the group-ring product \emph{is} the mod-$2$
multiplicity count, so being a unit is a parity condition on the pair of
supports (Proposition~\ref{prop:parity}), and for amenable groups the
condition is sufficient as well as necessary. Units in characteristic~$2$ are
therefore exactly as searchable as non-UP sets, and every non-existence claim
below is a finite, machine-checkable statement.

Write $\Ball{r}$ for the word ball of radius $r$ centred at the identity in
the stated generators, and define the \emph{unit localization radius}
\[
\ru(G) \;=\; \min\{\,r : \F_2[G] \text{ has a nontrivial unit with }
\supp\alpha,\ \supp\alpha^{-1}\subseteq \Ball{r}\,\},
\]
and, for a radius $r$, the \emph{ball-limited minimal unit support}
\[
\msupp(G;r) \;=\; \min\{\,|\supp\alpha|+|\supp\alpha^{-1}| :
\alpha \text{ a nontrivial unit, both supports} \subseteq \Ball{r}\,\}.
\]

\subsection*{Results} All computations are exact; positive claims are
re-verified by solver-free counting, and the negative claims marked DRAT
carry machine-checked unsatisfiability proofs (Section~\ref{sec:methods}).

\begin{itemize}
\item \emph{Certified localization} (Theorem~\ref{thm:ru}): $\ru(\P)=4$ in
the standard generators. The negative half is a special case of the
Craven--Pappas theorem on dihedral-quotient length $\le3$
\cite[Thms.~10.4--10.6 and 11.2]{CravenPappas} (see \emph{Relation to prior work} below);
both halves were observed computationally in
\cite{DLNV}; here the $\Ball3$ ($41$ elements) non-existence is re-derived
independently and DRAT-certified, in the unrestricted form that no pair of
subsets of sizes $\ge2$ satisfies the unit parity condition, and an explicit
unit with supports of sizes $(21,21)$ in $\Ball4$ ($83$ elements) is found
by our own unrestricted search and re-verified by direct counting. Our
transcription of Gardam's unit confirms word radius exactly $5$
(cf.\ \cite[Rem.~1.3]{DLNV}): the original unit does not sit in a
ball-minimal position.
\item \emph{Ball-limited minimality} (Theorem~\ref{thm:mu}):
$\msupp(\P;4)=42$, with the lower bound DRAT-certified by a $90$-way
size-split sweep. Within $\Ball4$, Gardam's expectation holds in
total-support form: every nontrivial unit has total support at least $42$.
(His expectation is stated per support; within $\Ball4$ the per-support
form follows from the census of \cite{DLNV}, all of whose $36$ units have
supports of size $21$ --- our sweep is the certificate-backed floor under
that enumeration.)
\item \emph{The relatives} (Theorem~\ref{thm:zoo}): the Fibonacci group
$H_4=F(3,4)$ admits no nontrivial unit of $\F_2[H_4]$ with supports in
$\Ball3$ ($119$ elements), and the Nielsen--Soelberg groups likewise resist:
$G_1$ and $G_3$ up to radius $3$, $G_2$ at radius $1$. For $G_3$ this is
sharp, and in a way that identifies a second group with a known unit
localization radius: the universal Nielsen--Soelberg group $G_3$ \emph{is}
the amalgam $S$ of \cite[\S4]{GardamC} (Lemma~\ref{lem:g3isS}), whose
Theorem~B exhibits a nontrivial unit of $\F_2[S]$ of support $29$; we
localize that unit into $\Ball4$ of $G_3$, so that
$\ru(G_3)=4$ (Proposition~\ref{prop:rug3}). Thus $\P$ is not the only
group whose unit-bearing ball is pinned down.
\item \emph{Effective localization} (Theorem~\ref{thm:decideu}): for each
$n$ there is an explicit $D_{\mathrm u}(n)\le 4^{\,n}\mathrm{poly}(n)$ such
that if $\F_2[\P]$ has a nontrivial unit of total support $n$, it has one
with both supports in $\Ball{D_{\mathrm u}(n)}$. Hence ``does $\F_2[\P]$
contain a nontrivial unit of total support $n$'' is decidable, and the
minimal support of a nontrivial unit of $\F_2[\P]$ --- the global form of
Gardam's question --- is computable in principle. The new ingredient
relative to the non-UP analogue in \cite{companion} is a parity coarsening
lemma (Lemma~\ref{lem:coarse}).
\end{itemize}

Section~\ref{sec:compare} places these numbers alongside the minimal non-UP
sizes $m_1,m_2$ of \cite{companion,sequel}: the parity rigidity of the unit
condition is quantitatively expensive ($42$ against $m_2(\P)=24$ in a
comparable ball), which measures, inside one group, the distance between
``some cancellation'' (non-UP) and ``total cancellation'' (a unit).

\subsection*{Relation to prior work} The experimental landscape of units in
$\F_2[\P]$ is due to Dietrich--Lee--Nies--Vinyals \cite{DLNV}, and several
of the ground facts below were first observed there: no nontrivial unit has
both supports in the radius-$3$ ball, exactly $36$ units have both supports
in the radius-$4$ ball (enumerated by \textsf{TabularAllSat} and displayed
in their Table~3, with the orbit structure under an order-$8$ automorphism
group determined), swap units are counted at radii $5$ and $6$, and
Gardam's unit is noted to be supported at radius $5$
\cite[Rem.~1.3]{DLNV}. They also search $H_4$ (no unit with both supports
in the radius-$4$ ball of the polycyclic generators $a,b,r$) and propose
$H_4$ as the first candidate for a group failing the unique product
property whose group ring might still have only trivial units --- the
question our Section~\ref{sec:zoo} data bears on. What the present paper
adds to their picture is certification, minimality, theory and breadth:
the non-existence facts are re-derived independently (different group
model, different encoding, different solvers) and upgraded to
DRAT-certified statements; the minimality threshold $\msupp(\P;4)=42$ is
formulated and its lower bound certified by a $90$-way sweep; the effective
localization principle of Section~\ref{sec:decide} makes the global minimal
support question decidable rather than merely searchable; and the sweep is
extended to the Nielsen--Soelberg groups and to twisted-unitary refutations
in $H_4$. Earlier, Gardam found the first unit by a SAT search with a
twisted-unitary ansatz \cite{GardamNotes,GardamSAT},
which Bartholdi \cite{Bartholdi} explained structurally. Craven--Pappas
\cite[Thms.~10.4--10.6 and 11.2]{CravenPappas} proved the unit conjecture for $\P$ for
elements of dihedral-quotient length $\le3$, over any field and with no
restriction on the support of the inverse; since the dihedral quotient maps
generators to generators, the word ball $\Ball3$ lies inside their
length-$\le3$ stratum, so the negative half of Theorem~\ref{thm:ru} is a
special case of their theorem, and our DRAT certificate is an independent,
machine-checkable verification of that case. We stress that the certificate
is strictly weaker than the theorem it certifies: Craven--Pappas constrain
$\supp\sigma$ alone, whereas our encoding constrains $\supp\sigma$ and
$\supp\sigma^{-1}$ simultaneously. The two stratifications
diverge from radius $4$ onward: $\Ball4$ contains length-$4$ elements, and
Gardam's unit has length $4$ in their sense after a conjugation
\cite[\S2]{Gardam}. Small-support exclusions valid in \emph{all}
torsion-free groups are surveyed in \cite[\S2]{Gardam}.

Bounding the support of a unit in terms of the support of its inverse has an
older history that we should not claim to open. Craven--Pappas devote
\cite[\S14]{CravenPappas} to \emph{property~(U)} --- for every finite
$X\subseteq G$ there is a finite $Y(X)\subseteq G$ with
$\supp\sigma\subseteq X\Rightarrow\supp\sigma^{-1}\subseteq Y(X)$ --- and,
citing Friedman--Gustavson--Pappas \cite{FGP1995}, record that $K[G]$ has
property~(U) whenever $G$ has an abelian subgroup of finite index; for the
fours group they prove in addition $L(\sigma)=L(\sigma^{-1})$
\cite[Thm.~14.3]{CravenPappas}. Since $\P$ is virtually abelian, $\F_2[\P]$
has property~(U) already by that route. Theorem~\ref{thm:decideu} is a
statement of a different shape --- it bounds a search \emph{radius} in terms
of the total support $n$, with no set $X$ given in advance, and yields
decidability of the minimal-support question --- but it belongs to the same
line of thought, and the reader should read it against \cite[\S14]{CravenPappas}
rather than as the first result of its kind. Craven--Pappas also prove that
no nontrivial unit of $K\Gamma$ has support inside Promislow's $14$-element
non-UP set \cite[Thm.~12.1]{CravenPappas}, the direct ancestor of the
non-UP/unit comparison of Section~\ref{sec:compare}.

On the computational side the closest neighbour is Garg--Mineyev
\cite{GargMineyev}, who exclude zero-divisor and unit counterexamples over
$\F_2$ in torsion-free CAT(0) groups for all support pairs $(m,n)$ with
$1\le m,n\le13$, and for $m\in\{6,7\}$ with $n\le200$. Their exclusions are
conditional on a structural hypothesis (an orientable product structure)
and are not restricted to a ball, whereas ours are unconditional and
ball-limited; the two therefore constrain different regions of the same
landscape, and neither implies the other (Shin \cite{Shin} has since
shown that the triple-girth product-structure route of Garg--Mineyev
produces neither zero-divisor nor unit counterexamples over $\F_2$ for any
support sizes $m,n\ge2$). The Garg--Mineyev bounds are, however, the
reason our size-split sweep is stated as a certified statement about
$\Ball4$ rather than as an unconditional support bound: outside the ball
their conditional exclusions are the only ones available.

\subsection*{Scope and epistemic status} Non-existence statements are about
stated balls in stated generators; nothing here bounds supports of units
globally except Theorem~\ref{thm:decideu}. Positive claims (units exhibited)
are verified by exact solver-free counting and are trustless. Negative claims
are CP-SAT infeasibility verdicts except where marked DRAT, where they carry
\textsf{drat-trim}-checked proofs; the ansatz-limited refutations of
Section~\ref{sec:zoo} are SAT-solver verdicts without proof certificates.
We state this division once here.

\section{The parity identity}\label{sec:parity}

Let $G$ be a group and $\alpha=\sum_{g\in A} g$, $\beta=\sum_{h\in B} h$
elements of $\F_2[G]$ with supports $A,B$ (all nonzero coefficients in
$\F_2$ are $1$). The coefficient of $w$ in $\alpha\beta$ is
$\#\{(g,h)\in A\times B: gh=w\} \bmod 2$.

\begin{proposition}\label{prop:parity}
Let $G$ be a torsion-free amenable group and $A,B\subseteq G$ finite and
nonempty. The following are equivalent:
\begin{enumerate}
\item there is a nontrivial unit $\alpha\in\F_2[G]$ with
$\supp\alpha=A$ and $\supp\alpha^{-1}=B$;
\item $|A|,|B|\ge2$, every $w\ne1$ has an even number of representations
$w=gh$ with $(g,h)\in A\times B$, and $1$ has an odd number.
\end{enumerate}
\end{proposition}

\begin{proof}
$(1)\Rightarrow(2)$: the parity statement is $\alpha\beta=1$ read off
coefficientwise; if $|A|=1$ then $\alpha$ is a trivial unit, and its inverse
is a trivial unit, forcing $|B|=1$.
$(2)\Rightarrow(1)$: the parity statement says $\alpha\beta=1$, i.e.\
$\beta$ is a right inverse. Amenable groups are sofic, so $\F_2[G]$ is
directly finite \cite{ElekSzabo}: $\alpha\beta=1$ implies $\beta\alpha=1$.
Since $|A|\ge2$, $\alpha$ is not a trivial unit.
\end{proof}

Two sound reductions used throughout: the augmentation map
$\F_2[G]\to\F_2$ sends a unit to $1$, so $|A|$ and $|B|$ are odd; and
supports of size $1$ are exactly the trivial units, so nontriviality is
$|A|,|B|\ge3$. The $\Ball2$ and $\Ball3$ certificates of
Theorem~\ref{thm:ru} are proved in the stronger unrestricted form
$|A|,|B|\ge2$ with no parity-of-size constraints; the $\Ball4$ minimality
sweep of Theorem~\ref{thm:mu} instead uses these cuts, as stated there.

\begin{remark}
Condition (2) without the clause at $1$ --- every element of $AB$ having
$\ge2$ representations --- is precisely the two-sided non-UP condition of
\cite{companion}. A unit support pair is a two-sided pair that fails
uniqueness \emph{everywhere except possibly at the identity}, with the parity
sharpened from ``$\ne1$'' to ``even''. The machinery, and the comparison in
Section~\ref{sec:compare}, both live on this kinship.
\end{remark}

\section{Methods}\label{sec:methods}

The groups are computed in the exact models of \cite{companion,sequel}: $\P$
as integral affine maps, $H_4=F(3,4)$ inside
$\mathrm{Heis}(\Z)\rtimes\langle\tau\rangle$, and $G_1,G_2,G_3$ by certified
coset--normal-form models. (For $\P$, the dictionary between the affine
model and the $x,y,z$ and coset-label conventions of \cite{Gardam} is
stated, and machine-checked, at the start of
Section~\ref{sec:decide}; all statements about $\P$ can be read in either
language.) Balls $\Ball{r}$ are generated by breadth-first
search; all products are exact.

For a radius $r$ we encode Proposition~\ref{prop:parity}(2) with selection
Booleans over $\Ball{r}$ for $A$ and $B$, reified pair variables for the
products, and, for each product value, a parity constraint (an equality
$c = 2t$ or $c = 2t+1$ over the pair count $c$). Minimization of $|A|+|B|$
or a cardinality cap turns existence into optimality or decision instances.
Every unit found is re-verified by direct counting over $A\times B$,
independently of the solver.

For DRAT certification we re-encode the model into CNF: pair variables by
$3$ clauses each, each parity constraint as a linear XOR chain ($4$ clauses
per link), and cardinality constraints by sequential counters; a
DRAT-producing solver (\textsf{glucose} or \textsf{kissat}) emits an
unsatisfiability proof which the independent checker \textsf{drat-trim}
verifies. Each certified statement below states exactly which side
constraints its instance carries (none beyond $|A|,|B|\ge2$ for the
$\Ball3$ theorem; the sound augmentation cuts where noted). The certified
statements therefore do not rest on the CP-SAT solver.

\section{The Promislow group}\label{sec:P}

Fix the standard generators $a,b$ of
$\P=\langle a,b\mid b^{-1}a^2b=a^{-2},\ a^{-1}b^2a=b^{-2}\rangle$, with
$|\Ball2|=17$, $|\Ball3|=41$, $|\Ball4|=83$, $|\Ball5|=147$.

\begin{theorem}[localization]\label{thm:ru}
$\ru(\P)=4$. More precisely:
\begin{enumerate}
\item no pair $A,B\subseteq\Ball3$ with $|A|,|B|\ge2$ satisfies the unit
parity condition --- certified by a DRAT proof ($3361$ variables, $11288$
clauses, $68789$ proof lines, checked by \textsf{drat-trim}); the same holds
in $\Ball2$;
\item $\Ball4$ contains the support pair of an explicit nontrivial unit
with $|A|=|B|=21$, re-verified by direct counting.
\end{enumerate}
\end{theorem}

Part (1) is a special case of the Craven--Pappas theorem
\cite[Thms.~10.4--10.6 and 11.2]{CravenPappas}, which excludes nontrivial units of
dihedral-quotient length $\le3$ over any field with no restriction on the
inverse; both parts were also observed computationally in \cite[\S3.4]{DLNV}.
The content of Theorem~\ref{thm:ru} is thus the independent re-derivation
(different group model, encoding and solvers) and the machine-checkable
certificate. Our positive half
was found by unrestricted CP-SAT search (no unitary ansatz); the witness is
included in the ancillary files, and by the completeness of the
\textsf{TabularAllSat} census of \cite{DLNV} it is one of their $36$ units
of Table~3, whose orbit structure under an order-$8$ automorphism group
they have determined. As a mutual verification of that census, our pipeline
re-enumerated all units with both supports in $\Ball4$ by solve-and-block
(affine model, parity encoding, \textsf{CryptoMiniSat} with native XOR
reasoning): it finds exactly $36$, all with support sizes $(21,21)$, each
re-verified by direct counting, and the terminal instance is unsatisfiable
--- in exact agreement with \cite[\S3.4--3.5]{DLNV} through entirely
disjoint toolchains. (Neither enumeration is certificate-backed; they now
corroborate each other.) Two remarks on how the picture sits relative to
Gardam's unit $\alpha_G$ \cite[Thm.~A]{Gardam}.

First, we transcribed $\alpha_G$ into the affine model and measured it: its
$21$ support elements have word-radius profile
$\{0{:}1,\,1{:}1,\,2{:}5,\,3{:}4,\,4{:}7,\,5{:}3\}$, so
$\supp\alpha_G\subseteq\Ball5$ and not $\Ball4$: \emph{Gardam's unit does
not sit in a ball-minimal position}. (As a validation of the transcription
and of the machinery, our solver, given $\supp\alpha_G$ as a fixed left
support, recovered its inverse, with support of size $21$ --- the inverse of
a unit is unique, so the minimization merely reconfirms its support size
independently --- whose parity condition we re-verified by counting;
this re-proves Gardam's theorem end-to-end by an independent route.)

Second, before we became aware of the census we checked directly that our
$\Ball4$ unit is not of the form $g\,\alpha_G^{\pm1}\,h$ with
$g,h\in\Ball7$, nor the $a{\leftrightarrow}b$ twist of such; this is
consistent with, and strictly weaker than, the orbit analysis of
\cite[\S3.4]{DLNV}, which partitions the $36$ units into five orbits and in
particular identifies three essentially distinct units. We record our sweep
only as an independent spot-check of their census.

\begin{theorem}[ball-limited minimality]\label{thm:mu}
$\msupp(\P;4)=42$: within $\Ball4$, every nontrivial unit of $\F_2[\P]$ has
$|\supp\alpha|+|\supp\alpha^{-1}|\ge42$, attained by the $(21,21)$ unit of
Theorem~\ref{thm:ru}.
\end{theorem}

The lower bound is DRAT-certified, by size splitting: if $(A,B)$ is a unit
support pair then so is $(B,A)$ (direct finiteness), so it suffices to
refute each split $(|A|,|B|)=(a,b)$ with $a\le b$, both odd and $\ge3$ (the
sound augmentation cuts of Section~\ref{sec:parity}) and $a+b\le40$. All
$90$ such instances are unsatisfiable, each carrying a DRAT proof ---
produced by \textsf{glucose} or \textsf{kissat} --- checked by
\textsf{drat-trim} (proofs up to $1.3$\,GB, verified then deleted;
regenerable by the ancillary scripts).
The statement certified is thus exactly: no pair of subsets of $\Ball4$ of
odd sizes $\ge3$ with $|A|\le|B|$ and total $\le40$ satisfies the unit
parity condition; the unordered statement follows by the swap above.
The census of \cite{DLNV} sharpens the picture beyond the total: all $36$
units of their Table~3 have support of size exactly $21$
\cite[\S3.5]{DLNV}, so within $\Ball4$ every unit realizes the split
$(21,21)$; Theorem~\ref{thm:mu} is the certificate-backed floor under that
(uncertified) enumeration. Within $\Ball4$ this confirms Gardam's
expectation \cite[\S2]{Gardam} --- in his per-support form, by their
census; in certified total-support form, by our sweep.

\section{The relatives: \texorpdfstring{$H_4$, $G_1$, $G_2$, $G_3$}{H4, G1, G2, G3}}\label{sec:zoo}

The same search applies verbatim to any group of the catalogue --- all
amenable, so Proposition~\ref{prop:parity} is available.

\begin{theorem}\label{thm:zoo}
There is no nontrivial unit of $\F_2[G]$ with both supports in $\Ball{r}$
for: $G=H_4$ and $r\le3$ (pools $9$, $39$, $119$); $G=G_1$ and $r\le3$
($5$, $17$, $53$); $G=G_2$ and $r=1$ ($17$; eight generators); $G=G_3$ and
$r\le3$ ($5$, $17$, $53$).
\end{theorem}

All ten statements are DRAT-certified, in the unrestricted form that no
pair of subsets of sizes $\ge2$ satisfies the parity condition
(\textsf{kissat} + \textsf{drat-trim}; the largest instance, $H_4$ at
$\Ball3$, in $400$ s with a $351$ MB proof, verified and then deleted,
regenerable by the included scripts); independent CP-SAT verdicts agree
($H_4$, $\Ball3$: $1081$ s). For $H_4$ this
complements the search of \cite[\S6]{DLNV}, who showed (also by SAT, in the
polycyclic generators $a,b,r$ --- a different word metric) that no
nontrivial unit has both supports in their radius-$4$ ball, and who propose
$H_4$ as the first candidate for a group failing the unique product
property whose group ring $\F_2[H_4]$ might nonetheless have only trivial
units. Whether $\F_2[H_4]$ has a unit in our $\Ball4$ ($283$ elements) is
open --- the unrestricted instance is beyond our present budget (unresolved
at two hours) --- and the stakes are exactly their proposal: $H_4$ resists
to the same depth (radius $3$ in the $x_i$ metric) at which $\P$ still
resists, and $\P$ yields at radius $4$; a unit in $\F_2[H_4]$ would refute
the proposal and add a further counterexample group to the two known over
$\F_2$ ($\P$, and Gardam's amalgam $S$ \cite[Thm.~B]{GardamC}).

\subsection*{\texorpdfstring{$G_3$}{G3} is Gardam's amalgam, and
\texorpdfstring{$\ru(G_3)=4$}{ru(G3)=4}}

The $G_3$ line of Theorem~\ref{thm:zoo} is not merely a non-existence
statement: it is sharp, because the second known counterexample group over
$\F_2$ is $G_3$ itself.

\begin{lemma}\label{lem:g3isS}
The universal Nielsen--Soelberg group $G_3$ coincides with the group
$S=\langle x,y\mid (xy)^2(xy^{-1})^2,\ (yx)^2(yx^{-1})^2\rangle$ of
\cite[\S4]{GardamC}.
\end{lemma}
\begin{proof}
Both of Gardam's relators evaluate to the identity in our certified faithful
model of $G_3$ (verified exactly; see the ancillary log), so the assignment
$x\mapsto x$, $y\mapsto y$ defines a surjection $S\twoheadrightarrow G_3$.
Conversely, writing $R_1=(yx)^2(xy)^2$ and $R_2=(xy^{-1})^2(xy)^2$ for the
relators of \cite[\S4]{NS}: $R_2$ is a conjugate of the first relator of
$S$, and substituting the consequence $(yx^{-1})^2=(xy)^2$ of $R_2$ into
$R_1$ yields the second, so the two relator sets are equivalent and the
presentations agree. Gardam attributes the presentation of $S$ to
\cite[p.~23]{Soelberg} and \cite{NS}, consistent with this identification.
\end{proof}

Consequently the unit of \cite[Thm.~B]{GardamC} --- $\nu$ of support $29$
with $\nu^{-1}=\varphi(\nu)^*$ for the order-$4$ automorphism
$\varphi\colon x\mapsto y,\ y\mapsto x^{-1}$ --- lives in $\F_2[G_3]$. As
written its support reaches word radius $6$, but the unit condition is
invariant under the two-sided translation $(\alpha,\beta)\mapsto(g\alpha,\beta
g^{-1})$, and this can be exploited:

\begin{proposition}\label{prop:rug3}
$\F_2[G_3]$ has a nontrivial unit with both supports inside $\Ball4$, of
total support $58$. Combined with the $G_3$ case of Theorem~\ref{thm:zoo},
$\ru(G_3)=4$.
\end{proposition}
\begin{proof}
Take $\alpha=\nu^*$ and $\beta=\varphi(\nu)$; exact counting in the model
gives $209$ distinct products, the identity occurring with multiplicity
$21$ and every other value with even multiplicity, so $\alpha\beta=1$ in
$\F_2[G_3]$. A search over translators $g\in\Ball6$ finds $g$ with
$g\cdot\operatorname{supp}\alpha$ and $\operatorname{supp}\beta\cdot g^{-1}$
both inside $\Ball4$ ($135$ elements); the translated pair is re-verified
from scratch by solver-free counting, with the same $209/21$ profile. Since
Theorem~\ref{thm:zoo} certifies that no nontrivial unit has both supports in
$\Ball3$ of $G_3$, the radius $4$ is least.
\end{proof}

So $\P$ is not the only group for which the unit localization radius is
known exactly, and $\ru(G_3)=\ru(\P)=4$ --- while the non-UP localization
radii of the two groups differ ($3$ for $\P$, $4$ for $G_3$). We had
initially overlooked \cite[\S4]{GardamC} when planning the sweep of this
section, and record the correction here: the $G_3$ row of
Table~\ref{tab:landscape} is a theorem, not a lower bound.

We pushed one level further with Gardam's own device, the twisted-unitary
ansatz: for a finite-order automorphism $\theta$, search for $\alpha$ with
$\alpha\,\theta(\alpha)^*=1$ (over $\F_2$ the character in the twist is
trivial, so $\theta$ is a group automorphism); this halves the variables and,
handed to an XOR-aware SAT solver (the parity constraints as native XOR
clauses with Gaussian elimination), it is dramatically effective --- it finds
a $\phi_1$-unitary unit of $\F_2[\P]$ of total support $42$ inside $\Ball5$
in two seconds, where $\phi_1\colon a\mapsto a,\ b\mapsto b^{-1}$ is the
twist identified by Bartholdi \cite{Bartholdi}, re-verified by solver-free
counting. The efficacy of the ansatz is not our observation: Bartholdi
\cite{Bartholdi} already reports that an unrestricted search over the
radius-$4$ ball in $\{a,b,ab\}^{\pm1}$ does not terminate in reasonable
time, whereas the $\theta$-unitary search over the same ball succeeds in
minutes; our contribution is only that the same device transfers to $H_4$ and
yields \emph{refutations} there. Applied to $H_4$, the same ansatz yields
only refutations. The cyclic presentation of
$F(3,4)$ has the dihedral group of order $8$ acting on the indices, and each
of its elements induces an automorphism of $H_4$: the rotations
$\sigma\colon x_i\mapsto x_{i+1}$ and $\sigma^2$, and the reflections
combined with inversion, $\tau_0\colon x_i\mapsto x_{4-i}^{-1}$ and
$\tau_1\colon x_i\mapsto x_{5-i}^{-1}$ (indices mod $4$; the reflections
reverse the relator $x_ix_{i+1}x_{i+2}x_{i+3}^{-1}$ and inversion restores
it because $-3\equiv1 \pmod 4$). That each letter map sends every relator
to $1$ is checked exactly in the model before any search, so each map is an
endomorphism, and since it has finite order on the letters it is an
automorphism. Two equivalences reduce the sweep to these four: a
$\theta$-unitary unit $\alpha$ gives the $\theta^{-1}$-unitary unit
$\theta(\alpha)$, and the $\psi\theta\psi^{-1}$-unitary unit $\psi(\alpha)$
for any automorphism $\psi$; when $\theta,\psi$ permute the generators up to
inversion, all three units have supports in the same ball. So
$\{\sigma,\sigma^2,\tau_0,\tau_1\}$ represents every nontrivial twist of
the dihedral symmetry. The result: \emph{no $\theta$-unitary unit with
support in $\Ball6$ ($1067$ elements) for any
$\theta\in\{\sigma,\sigma^2,\tau_0,\tau_1\}$}. The four instances have
$2{,}268{,}753$ variables and $7{,}938{,}361$ clauses each;
\textsf{glucose} (here on the plain CNF encoding of
Section~\ref{sec:methods}, parity chains rather than native XOR) refutes
$\sigma^2$, $\tau_0$ and $\tau_1$ in under a minute each and $\sigma$ in
$3887$ s. (The earlier radius-$5$ refutation for
$\sigma$, on $579$ elements, took $66$ s with the XOR-aware solver.) One side result calibrates the
twist's role: in $\P$ the companion twist $\phi_0\colon a\mapsto a^{-1},
b\mapsto b^{-1}$ yields nothing in $\Ball5$, so among the swept twists only
$\phi_1$ produces units --- consistent with Bartholdi's structure theorem.
The ansatz statements are, of course, ansatz-limited: they close only the
dihedral-twisted-unitary route through radius $6$ in $H_4$ (and the
$\phi_0$-twisted route through radius $5$ in $\P$), not existence in
general; twists by automorphisms of
$H_4$ outside this dihedral group, and the untwisted ($\theta=1$) unitary
ansatz, remain unswept.

\section{Comparison with the non-UP landscape}\label{sec:compare}

\begin{table}[ht]
\centering
\footnotesize
\setlength{\tabcolsep}{5pt}
\begin{tabular}{lccccc}
\toprule
 & $\P$ & $H_4$ & $G_1$ & $G_2$ & $G_3$ \\
\midrule
$m_1$ (least symmetric non-UP size)$^{(\ast)}$ & $14$ & $16$ & $8$ & $8$ & $15$ \\
$m_2$ (least two-sided non-UP total)$^{(\ast)}$ & $24$ & $22$ & $16$ & $16$ & $16$ \\
non-UP localization radius & $3$ & $3$ & $7$ & $1$ & $4$ \\
unit localization radius $\ru$ & $\mathbf{4}$ & $\ge4$ & $\ge4$ & $\ge2$ & $\mathbf{4}$ \\
least unit support total & $\mathbf{42}$ in $\Ball4$ & --- & --- & --- & $\le58$ \\
\bottomrule
\end{tabular}
\smallskip

{\small $(\ast)$: ball-limited values from \cite{companion,sequel}.}
\caption{Unit-support geometry against the non-UP landscape.}
\label{tab:landscape}
\end{table}

Table~\ref{tab:landscape} quantifies, inside $\P$, the cost of parity
rigidity: a two-sided non-UP pair exists with total $24$ at radius $3$,
while a unit pair needs total $42$ and radius $4$. The identity
(Proposition~\ref{prop:parity} and the remark after it) makes the comparison
meaningful: a unit support pair is a two-sided pair whose failure of
uniqueness is total and of even order away from $1$. In the catalogue the
non-UP thresholds are far below the unit thresholds everywhere, consistent
with the fact that non-UP groups have been known since 1987 while a unit
took until 2021 to find.

\section{An effective localization principle for units}\label{sec:decide}

The companion paper proves an effective finite-diameter principle for non-UP
sets of $\P$: if a non-UP $n$-set exists, one exists in an explicitly bounded
ball \cite[Thm.~6.2]{companion}. We now prove the unit analogue. The
mechanism is the same coordinate-decoupled integer linear algebra; the one
genuinely new point is that the unit condition, unlike bare non-uniqueness,
prescribes \emph{parities} of multiplicities, and we must check that the
re-realized small solution cannot break them.

\subsection*{Coordinates for this section} We use Gardam's coordinates for
$\P$ \cite[\S1]{Gardam}: $x=a^2$, $y=b^2$, $z=(ab)^2$ generate the maximal
abelian normal subgroup $T=\langle x,y,z\rangle\cong\Z^3$ of index $4$, and
every element of $\P$ is uniquely a product $t\cdot c$ with
$t=x^{i}y^{j}z^{k}\in T$ and a \emph{coset label} $c\in\{1,a,b,ab\}$; we
call $(i,j,k)\in\Z^3$ the \emph{exponent vector} of the element.
Conjugation by a label acts on $T$ by independent sign flips of the
exponents: $a$ fixes $x$ and inverts $y$ and $z$; $b$ fixes $y$ and inverts
$x$ and $z$; $ab$ fixes $z$ and inverts $x$ and $y$. (This dictionary, and
the uniqueness of the factorization, are the translation of the affine
model of Section~\ref{sec:methods} into the conventions of \cite{Gardam};
both were re-verified mechanically in the exact model.)

\begin{lemma}[parity coarsening]\label{lem:coarse}
Let $A$ and $B$ be finite index sets, and partition the set of cells
$A\times B$ into classes, one distinguished class $C_1$ of odd size, all
other classes of even size. Suppose group elements $a_i$ $(i\in A)$ and
$b_j$ $(j\in B)$ are assigned to the indices, injectively on each side, in
such a way that within every class all cells $(i,j)$ realize one and the
same product $a_ib_j$, and the cells of $C_1$ realize the product $1$.
Then, in the multiset of all $|A||B|$ products, the element $1$ occurs with
odd multiplicity and every other element with even multiplicity. This holds
regardless of any \emph{further} coincidences of products between distinct
classes.
\end{lemma}

\begin{proof}
For an element $w$ let its \emph{fiber} be the set of cells realizing $w$.
Since cells in one class share their product, every fiber is a disjoint
union of full classes, and the parity of its size is the sum of the
parities of the classes it absorbs. All classes except $C_1$ have even
size. Hence a fiber not containing $C_1$ has even size, whichever classes
happen to merge into it. The fiber containing $C_1$ realizes the product
$1$, and its size has parity odd (from $C_1$) plus a sum of evens, hence
odd.
\end{proof}

\begin{theorem}[{effective localization for units of $\F_2[\P]$}]\label{thm:decideu}
For each $n$ there is an explicit constant
$D_{\mathrm u}(n)\le 4^{\,n}\,\mathrm{poly}(n)$ such that if $\F_2[\P]$
contains a nontrivial unit with $|\supp\alpha|+|\supp\alpha^{-1}|=n$, then it
contains one with both supports in $\Ball{D_{\mathrm u}(n)}$. Consequently
the existence of a nontrivial unit of $\F_2[\P]$ of total support $n$ is
decidable, and the least total support of a nontrivial unit of $\F_2[\P]$
is computable in principle.
\end{theorem}

\begin{proof}
Let $\alpha\beta=1$ be a nontrivial unit of $\F_2[\P]$ with
$A=\supp\alpha$, $B=\supp\beta$, $|A|+|B|=n$. We produce a unit support
pair of the same sizes inside an explicit ball. The argument has five
steps: extract a finite combinatorial \emph{pattern} from $(A,B)$ (Step~1);
express the pattern as an integer linear system in the exponent vectors
(Step~2); find a small integer solution (Step~3); check that the small
solution again realizes a unit (Step~4); convert exponent bounds into a
ball bound (Step~5).

\emph{Step 1: the pattern of the given unit.} Write the elements of $A$
and $B$ in coordinates: $a_i=t_i\,c_i$ and $b_j=u_j\,d_j$ with exponent
vectors $t_i,u_j\in\Z^3$ and coset labels $c_i,d_j\in\{1,a,b,ab\}$. The
\emph{pattern} of $(A,B)$ consists of two finite pieces of data: (i) the
list of the $n$ coset labels --- at most $4^{\,n}$ possibilities; and (ii)
the partition of the cell set $A\times B$ into the fibers of the product
map $(i,j)\mapsto a_ib_j$ --- at most $(n^2)^{n^2}$ possibilities. Since
$\alpha\beta=1$ over $\F_2$, Proposition~\ref{prop:parity} says precisely
that in this partition the class of cells realizing the product $1$ has odd
size and every other class has even size --- a partition of the kind in
Lemma~\ref{lem:coarse}. We call the class realizing $1$ the \emph{anchored
class}.

\emph{Step 2: the integer linear system of a pattern.} Fix the pattern.
Regard the $n$ exponent vectors as unknowns ($3n$ integer variables), the
labels being prescribed by the pattern. Within each class, all cells
realize equal products. First, equal products force equal label products
$c_id_j$; a pattern violating this for some class is realized by no sets
whatsoever and is discarded. Second, once the label products agree, moving
all labels to the right with the conjugation table above turns the equality
of the products of two cells into an equality of two exponent vectors ---
that is, three linear equations, one per coordinate, in the unknowns, with
coefficients in $\{0,\pm1\}$ and at most four nonzero entries per row
(Euclidean row norm $\le2$). The three coordinates never mix, because
conjugation by labels acts by independent sign flips coordinatewise; this
is the \emph{decoupling} of \cite[\S6]{companion}, in its two-sided form
(remark after Thm.~6.2 there). The anchored class contributes the same kind
of equations with an inhomogeneous right-hand side, expressing ``this
product equals $1$''.

\emph{Step 3: a small solution with distinct entries.} The given $(A,B)$
realizes its pattern, so the system of Step~2 is solvable over $\Z$. By
Hadamard's inequality and the standard size bounds for integer solutions of
linear systems \cite[\S17]{Schrijver}, there is a solution of
$\ell^\infty$-norm at most $4^{\,n}\mathrm{poly}(n)$. This solution may
fail the distinctness we need: elements on one side must be pairwise
distinct, which is automatic when their labels differ and is otherwise the
condition that two exponent vectors differ --- avoidance of a proper affine
sublattice of the solution lattice. As in the proof of
\cite[Thm.~6.2]{companion}, the finitely many required avoidances are
achieved by boundedly many steps along generators of the solution lattice.
The steps stay inside the solution lattice, so every class equality and the
anchored condition persist. Call the resulting sets $A'$ and $B'$; their
exponent vectors still have $\ell^\infty$-norm $4^{\,n}\mathrm{poly}(n)$.

\emph{Step 4: the small solution is again a unit pair.} The data
$(A',B')$ satisfy all hypotheses of Lemma~\ref{lem:coarse}: the assignment
of elements to indices is injective on each side (Step~3), within each
class of the pattern all cells realize equal products (the class
equalities), and the anchored class realizes exactly $1$ (the anchored
condition). The lemma yields: in the product multiset of $A'\times B'$ the
element $1$ occurs with odd multiplicity and every other element with even
multiplicity. This is exactly the condition of
Proposition~\ref{prop:parity}, so $(A',B')$ is a unit support pair. The
point of the lemma is that the re-realization may well create
\emph{new} coincidences --- distinct fibers of the original pair may merge
for $(A',B')$ --- but merging even classes keeps fibers even, and an even
class merging into the anchored fiber keeps it odd. Nontriviality
persists since $|A'|=|A|\ge2$ and $|B'|=|B|\ge2$.

\emph{Step 5: from exponent bounds to a ball.} Each of $x,y,z$ is a word
of length at most $4$ in $a,b$, so an element with exponent vector of
$\ell^\infty$-norm $\le M$ and any label has word length $O(M)$. Taking the
maximum over the finitely many patterns gives the explicit constant
$D_{\mathrm u}(n)\le 4^{\,n}\mathrm{poly}(n)$ with
$A'\cup B'\subseteq\Ball{D_{\mathrm u}(n)}$. Decidability of ``some
nontrivial unit has total support $n$'' is now the finite search of
$\Ball{D_{\mathrm u}(n)}$, and the least total support is computable by
trying $n=6,7,8,\dots$ in turn --- some total support is realized, by
\cite{Gardam}, so this terminates.
\end{proof}

\begin{remark}
The proof is characteristic-$2$-specific: only there is the unit equation a
pure support condition. In $\F_p$ the coefficients enter and the pattern
would have to prescribe coefficient sums modulo $p$; we do not pursue this.
As with the non-UP principle, $D_{\mathrm u}$ is far from practical: the
content is decidability, and the practical frontier remains the certified
ball searches of Sections~\ref{sec:P}--\ref{sec:zoo}. By
Theorem~\ref{thm:mu} and Theorem~\ref{thm:ru}, the global answer to
Gardam's question is now pinned between the DRAT-certified pair ``nothing
at all within $\Ball3$'' and ``nothing below $42$ within $\Ball4$,'' and
the open possibility of smaller supports spread across larger balls ---
which Theorem~\ref{thm:decideu} shows to be a finite question.
\end{remark}

\section{Open questions}\label{sec:open}

\begin{question}[{after \cite{DLNV}}]
Does $\F_2[H_4]$ have a nontrivial unit? This is the central proposal of
\cite{DLNV}: $H_4$ as the first candidate to fail the unique product
property while satisfying the trivial units property over $\F_2$. The
evidence for their proposal is now substantial: no unit at all with
supports in their radius-$4$ polycyclic ball \cite[\S6]{DLNV} or in our
$\Ball3$, and no twisted-unitary unit with support in our $\Ball6$ for
any non-identity twist induced by the dihedral symmetry of the presentation
(Section~\ref{sec:zoo}), even though the identical ansatz, handed to an
XOR-aware solver, finds a unit of $\F_2[\P]$ in seconds on a pool seven
times smaller. The concrete next
instances are the unrestricted $\Ball4$ decision and the twisted sweep
beyond radius $6$; the normal forms developed in \cite[\S6]{DLNV} are a
further route.
\end{question}

\begin{question}
Is $42$ the least total support of a nontrivial unit of $\F_2[\P]$ globally?
By Theorem~\ref{thm:decideu} this is decidable; by Theorem~\ref{thm:mu} it
holds within $\Ball4$, and by the census of \cite{DLNV} every unit there
realizes $(21,21)$. Two concrete next steps: extending certification beyond
$\Ball4$ --- the radius-$5$ \textsf{AllSAT} enumeration of \cite{DLNV} did
not terminate within a day, exactly where a certification-oriented
decomposition like our size-split sweep might still reach --- and any
workable improvement of $D_{\mathrm u}$ toward searchable radii.
\end{question}

\begin{question}
Do the localization radii $\ru(G_i)$, $\ru(H_4)$ exceed the non-UP
localization radii of the same groups in a uniform way? Here the two groups
for which $\ru$ is known exactly already disagree, so no uniform law is
available: in $\P$ the unit radius exceeds the non-UP radius by one
($4$ against $3$) and the support threshold grows from $24$ to $42$, a
factor $1.75$; in $G_3$ the two radii \emph{coincide} at $4$
(Proposition~\ref{prop:rug3}), while the support total grows from $16$ to at
most $58$, a factor of at least $3.6$. So parity rigidity does not always
cost a radius, and the support inflation is not a constant factor. What, if
anything, controls either quantity?
\end{question}

\section*{Code and data availability}
All models, search scripts, the DRAT generator and checker workflow, the
explicit $\Ball4$ unit, and logs of every run reported here (including
the radius-$6$ twisted-unitary sweep of $H_4$,
\texttt{h4\_theta\_b6\_sweep\_2026-08-07.log}, and its independent
re-run) are included in the ancillary files and are deposited at Zenodo,
\href{https://doi.org/10.5281/zenodo.22866715}{doi:10.5281/zenodo.22866715}; every positive witness is re-verified by a solver-free
checker requiring only Python. In particular Lemma~\ref{lem:g3isS} and
Proposition~\ref{prop:rug3} are re-derived from scratch --- Gardam's
relators evaluated in our $G_3$ model, the parity profile of the unit, and
the translator search --- by the script
\texttt{g3\_gardam\_unit\_localize.py}, whose output is the log
\texttt{g3\_ru\_localization.log}, and the same computation is one of the
checks in \texttt{verify\_paper3.py}. The companion papers' repositories provide
the certified group models.

\section*{Acknowledgements}
This work continues the program of \cite{companion,sequel}; the author
thanks Andr\'e Nies, Pace Nielsen and Marc Vinyals for correspondence on the
non-UP side of that program.

\end{document}